\documentclass[11pt,a4paper,reqno]{amsart}
\newcommand{\ez} {\hspace*{\fill} $\square$}
\usepackage{amsmath,amssymb,amsfonts,amsthm,mathtools}
\usepackage{enumerate}
\usepackage[colorlinks=true,citecolor=black,pagebackref=false]{hyperref}

\numberwithin{equation}{section}
\allowdisplaybreaks

\newcommand{\al}{\alpha}
\newcommand{\E}{{\mathcal E}}
\newcommand{\R}{\mathbb R}
\newcommand{\N}{\mathbb N}

\newcommand{\la}{\lambda}
\newcommand{\La}{\Lambda}
\newcommand{\A}{{\mathcal A}}
\newcommand{\Z}{{\mathbb Z}}

\newcommand{\Om}{\Omega}
\newcommand{\D}{{\mathcal D}}
\newcommand{\cC}{{\mathcal C}}

\makeatletter
\@namedef{subjclassname@2020}{\textup{2020} Mathematics Subject Classification}
\makeatother
\newtheorem{thm}{Theorem}[section]
\newtheorem{prop}[thm]{Proposition}
\newtheorem{lem}[thm]{Lemma}
\newtheorem{coro}[thm]{Corollary}
\newtheorem{defi}[thm]{Definition}
\newtheorem{rem}[thm]{Remark}
\newtheorem{conj}[thm]{Conjecture}
\newtheorem{ques}[thm]{Question}
\newtheorem{exam}[thm]{Example}

\title[]
{Spectral eigenvalue problem of Cantor measures and Artin's primitive root conjecture}

\author[X.-G. He]{Xing-Gang He}

 \address[Xing-Gang He]{School of Mathematics and Statistics, Hubei Key Laboratory of Mathematical Sciences, Central China Normal University, Wuhan 430079, P. R. China}
\email{xingganghe@163.com}

\author[Z.-Y. Wu]{Zhi-Yi Wu}
\address[Zhi-Yi Wu]{School of Mathematics and Information Science, Guangzhou University, Guangzhou, 510006, P.~R.~China}
\email{zhiyi\_wu2021@163.com}

\author[F.-L. Yin]{Feng-Li Yin}
\address[Feng-Li Yin]{School of Mathematics and Statistics, Zhoukou Normal University, Zhoukou 466001, P. R.
China}
\email{yinfengli05181@163.com}

\thanks{This work was supported by the National Natural Science Foundation of China 12371087,
12301105 and 12401112.}

\begin{document}
\begin{abstract}
 The eigenvalue problem for a probability measure $\mu$ with compact support in $\R$ is whether there exist
 a countable set $\Lambda$ and a nonzero real $t\ne 1$ such that both $\Lambda$ and $t\Lambda$ are
 spectra of $\mu$, that is, the family
$$E_{a\Lambda}=\{e^{-2\pi i a\lambda x}:\lambda\in\Lambda\}$$ is an orthonormal base for $L^2(\mu)$ for $a=1, t$. The
eigenvalue problem  was discovered independently by Strichartz \cite{Str00},  {\L}aba and Wang
\cite{LW02} for the Cantor measures $\mu_{4,\{0,1\}}$ and $\mu_{6,\{0,1,2\}}$, respectively.
 In this paper, we investigate the spectral eigenvalue problem for the general spectral Cantor
 measure $\mu_{b,\mathcal{D}}$. This topic is naturally related to elementary number theory.
 Unexpectedly, however, our main results depend on the theory of integers,
especially Artin's primitive root conjecture. To some extent, our results suggest
that Artin's primitive root conjecture may hold and confirms some viewpoints implied by Minkowski
in \cite{Min57}.
\end{abstract}

\maketitle

\section{Introduction}
Let $\mu$ be a Borel probability measure with compact support in $\R$ and let $\La\subset \R$ be a
countable set.  The pair $(\mu, \La)$ is called a {\it spectral pair} if the family
\[E_\La=\left\{e^{-2\pi i \la x}: \la\in\La\right\}\]
is an orthonormal basis of $L^2(\mu)$. In this case the $\mu$ is called a {\it spectral measure} and
the $\La$ is called a {\it spectrum} of $\mu$. The simplest example of a spectral pair is $(L|_{[0,
1]}, \Z)$, where $L|_\Om$ denotes the Lebesgue measure restricting on the Borel set $\Om$ with
positive Lebesgue measure.  In general, most measures are not spectral. For example, {\L}aba
\cite{La01} proved that $L|_{I\cup J}$, where both $I$ and $J$ are closed intervals, is a spectral
measure if and only if $I\cup J$ can cover $\R$ by translations without overlaps, up to a Lebesgue
zero set. This topic is related to Fuglede conjecture. For relevant literature, see \cite{Fu74} and
the papers citing it. Spectral measure is closely connected to many physical systems, which can be
described via a Schr\"{o}dinger equation with an almost periodic potential \cite{BBM82}.

If $\mu$ is a spectral measure, then $\mu$ is a finite linear combination of Dirac measures with
equal distribution, or a Lebesgue measure restricting on a region $L|_\Om$ or singular continuous
measure $\mu$ (see \cite{HLL13,LW06}). It is surprising that there are huge differences between  the
analysis theory based on $L|_\Om$ and  those based on singular continuous measures $\mu$. Here we
study the following problem, which is impossible for $L|_\Om$,

\indent{\it Given a singular spectral pair $(\mu, \La)$, find all $t\in\R$ such that the $(\mu,
t\La)$ is also a spectral pair.}

This problem is termed the {\it spectral eigenvalue problem}, also the {\it scaling spectrum
problem}, of the spectral pair $(\mu, \La)$. The names stem from the fact that we are determining for
which numbers $t$ the scaled set $t\La$ remains a spectrum, which is equivalent to the exponential
system $E_{t\La}$ being both orthogonal and complete in $L^2(\mu)$.

Such a problem, which is both intriguing and interesting, was initiated by Strichartz \cite{Str00}
and by {\L}aba and Wang \cite{LW06}, for some Cantor measures. The general Cantor measures $\mu_{b,
\D}$ are decided by a number $b$ with $|b|>1$ and a finite digit set $\D$ in $\R$. The measure
$\mu_{b, \D}$, also called a {\it self-similar measure},  is a Borel probability measure with compact
support
\begin{eqnarray*}
  T(b, \D)=\left\{\sum_{k=1}^\infty d_kb^{-k}: \text{all $d_k\in\D$}\right\}:=\sum_{k=1}^\infty \D
  b^{-k},
\end{eqnarray*}
and is the unique one satisfying that
\begin{eqnarray*}
  \mu_{b, \D}(E)=\frac 1{\#\D}\sum_{d\in\D}\mu_{b, \D}(\phi_d^{-1}(E)),\qquad \text{for any Borel set
  $E\subset\R$,}
\end{eqnarray*}
 where $\# S$ is the cardinality of the set $S\subset \R$,  $\{\phi_d(x)=\frac 1b(x+d)\}_{d\in\D}$ is
 an iterated function system, see, e.g., \cite{Fa90}.

Since the pioneering work of Jorgenson and Pedersen \cite{JP98} in 1998, i.e., the fourth middle
Cantor measure $\mu_{4,\{0,2\}}$ is a spectral measure, there has been extensive research on the
spectrality and spectral structure of self-similar measures, see
\cite{Dai12,Dai16,DHL13,DHS09,DH16,DK18,FHW18,HTW19,LW02} and the references therein. Many strange
properties of spectral self-similar measures,  unlike those of spectral Lebesgue measures, are
discovered. For example, the spectra of some known spectral measures $\mu_{b, \D}$ with $\#\D<|b|$ up
to translations have the cardinality of the continuum (see, e.g., \cite{DHL13, DHS09}). It is
conjectured that this is true for any singular continuous spectral measure.

  The fundamental result on the topic in $\R$ was obtained by {\L}aba and Wang \cite{LW02}. To
  introduce their work we need the following notation:
 \begin{defi}
 Let $b\in \Z$ with $|b|>1$ and let $\D\subset \Z$ be a finite set. We say $(b, \D)$ is an {\it
 integrally admissible pair} if there exists $\cC$ with $0\in\cC\subset\Z$ and $\#\D=\#\cC$ such that
 the matrix
 \[H=\frac 1{\sqrt{\#\D}}\left[e^{-2\pi i\frac{cd}{b}}\right]_{d\in\D, c\in \cC}\]
 is unitary, i.e., $HH^\ast=I$. In this case, we call the $(b, \D, \cC)$ an {\it integral Hadamard
 triple}.
 \end{defi}

 \begin{thm}[\cite{LW02}]
 Let $(b, \D, \cC)$ be an integral Hadamard triple. Then the self-similar measure $\mu_{b, \D}$ is a
 spectral measure. Moreover, if $\gcd(\D)=1$ and $\cC\subset b[-1/2, 1/2)\cap \Z$, then $(\mu_{b,
 \D}, \La(b, \cC))$ is a spectral pair, where
 \begin{eqnarray*}
   \La(b, \cC)=\bigcup_{k=1}^\infty \La_k(b, \cC), \qquad \La_k(b, \cC)=\cC+b\cC+\cdots+b^{k-1}\cC.
 \end{eqnarray*}
 \end{thm}
 In this paper we focus on the spectral eigenvalue problem of spectral measures $\mu_{b, \D}$ if
 $b>2$ and $(b, \D)$ is an  integrally admissible pair with $b>\#\D$. In order to get our results we
 need to choose another spectrum $\La(b, \cC)$, where $\cC$ satisfies that
 \begin{eqnarray*}
   0\in\cC\subset\{0, 1, \ldots, b-2\}.
 \end{eqnarray*}
We will prove the existence of such spectrum and call it a {\it canonical spectrum} of $\mu_{b, \D}$, call $(b, \D, \cC)$ a {\it canonical Hadamard triple} and $(\mu_{b, \D}, \La(b, \cC))$ a {\it
canonical spectral pair}.

 Unexpectedly, the spectral eigenvalue problem is essentially related to the Fourier analysis built
 on Cantor measure $\mu_{4, \{0, 2\}}$. Jorgenson and Pederson \cite{JP98} proved that $(\mu_{4, \{0,
 2\}}, \La(4, \{0, 1\})$ is a spectral pair. Strichartz \cite{Str06} proved that the mock Fourier
 series of $f\in C([0, 1])$ with respect to (w.r.t.) the spectral pair $(\mu_{4, \{0, 2\}}, \La(4,
 \{0, 1\})$ is uniformly convergent on the Cantor set $T(4, \{0, 1\})$. Whereas Dutkay, Han and Sun
 \cite{DHS14} showed that there exists $f\in C([0, 1])$ such that its mock Fourier series  w.r.t.
 the spectral pair $(\mu_{4, \{0, 2\}}, 17\La(4, \{0, 1\})$ is divergent at $0$.

It is known that all spectra of $\mu_{b, \D}$ containing zero are in $\Z$ if $(b, \D)$ is an
integrally admissible pair with $\gcd(\D)=1$ \cite{DHL19}. And $\La$ is a spectrum if and only if
$-\La$ is a spectrum. Then to investigate  the spectral eigenvalue problem of the canonical spectral
pair $(\mu_{b, \D}, \La(b, \cC))$  we only need to consider
\[t\in\N=\{1, 2, \ldots\}.\]
Therefore we call the spectral eigenvalue problem  the {\it spectral integer eigenvalue problem} in
this case.

To state out results we begin with a general conjecture:
\begin{conj}\label{conj1}
  Let $(\mu, \La)$ be a singular spectral pair in $\R$. Then there are infinitely many spectral
  eigenvalues  and infinitely many non spectral eigenvalues.
\end{conj}
If this conjecture holds true, a fundamental question arises concerning its interaction with
elementary number theory.
\begin{ques}\label{ques1}
Let $p$ be a prime. Is $p^n$ a spectral integer eigenvalue for the spectral pair $(\mu, \Lambda)$ for
all $n\geq1$?
\end{ques}

On this topic the first work was obtained  by Dutkay and  Hausserman \cite{DH16}, who study the spectral integer eigenvalue problem of the pair  $(\mu_{4,\{0,2\}},{}\allowbreak \Lambda(4,\{0,1\}))$. Precisely, they proved that both Conjecture \ref{conj1} and Question \ref{ques1} (for all primes $p$ with $p>3$) are true for $(\mu_{4, \{0, 2\}}, \Lambda(4, \{0, 1\}))$. These were later generalized to $(\mu_{2q, \{0, q\}}, \La(2q, \{0, 1\})$ \cite{DK18}. Motivated by their results, as well as by developments in spectral measure theory, particularly the work of {\L}aba and Wang \cite{LW02, LW06}, in this paper, we investigate the spectral integer eigenvalue problem for canonical spectral pairs $(\mu_{b,\mathcal{D}}, \Lambda(b, \mathcal{C}))$  focusing on Conjecture \ref{conj1} and Question \ref{ques1}.

As  Minkowski \cite{Min57} expressed in his preface to {\it Diophantische Approximationen}  that the
``deepest interrelationships in analysis are of an arithmetical nature",  the spectral integer
eigenvalue problem of canonical spectral pairs is closely related to the order of $b$ modulo $t$, and
thereby to the celebrated primitive root conjecture posed by E. Artin \cite{artin65} in 1927, which
is the focal point of diverse areas of mathematics such as group theory, algebra, analytic number
theory, and algebraic geometry.

Let $(b, \D, \cC)$ be a canonical Hadamard triple and $t\in\N$. We first prove the following facts
(some of which are straightforward and analogous to those in \cite{DH16}):
\begin{itemize}
  \item If $t$ is not a spectral eigenvalue, then so is $kt$ for $k\in\N$;
  \item If $t$ is not a spectral eigenvalue and all proper factors of $t$ are spectral eigenvalues,
      then $\gcd(t, b)=1$;
  \item $b-1$ is not a spectral eigenvalue;
  \item If $p_1, p_2, \ldots, p_m$ are all distinct prime factors of $b$ and $t$ is a spectral
      eigenvalue, then so is $tp_1^{\al_1}p_2^{\al_2}\cdots p_m^{\al_m}$ for each
      $\al_k\in\Z_{\geq0}$ with $1\le k\le m$. Here and in the sequel, we denote $\Z_{\geq
      n}=\{n,n+1,\ldots\}$ for $n\in\Z$.
\end{itemize}
From those  we begin with the following definition:
\begin{defi}
Let $(b, \D, \cC)$ be a canonical Hadamard triple and $t\in\N$. $t$ is called a primitively non
spectral integer eigenvalue of $(\mu_{b, \D}, \La(b, \cC))$ if $t$ is not a spectral eigenvalue and
all proper factors $t'$ of $t$ (i.e., $t'\mid t$ and $t'\neq t$) are spectral eigenvalues. And $t$ is
called a primitively spectral integer eigenvalue of $(\mu_{b, \D}, \La(b, \cC))$ if it is a spectral
eigenvalue with $\gcd(t, b)=1$.
\end{defi}
Hence, to study the spectral integer eigenvalue problem in this case it is sufficient to investigate
the primitive cases. We first prove the following result:
\begin{thm}
  There are infinitely many primitively non spectral integer eigenvalues of the canonical spectral
  pair $(\mu_{b, \D}, \La(b, \cC))$.
\end{thm}
The complementary case, however, is still unresolved, leading to the following conjecture.
\begin{conj}\label{conj2}
  There are infinitely many primitively spectral integer eigenvalues of the canonical spectral pair
  $(\mu_{b, \D}, \La(b, \cC))$.
\end{conj}

Let $t\in\N$ with $\gcd(t, b)=1$. As in number theory, the {\it order of $b$ modulo $t$} is defined
by
\begin{eqnarray*}
  O_b(t)=\min\{k\geq1: b^k\equiv1\pmod t\}.
\end{eqnarray*}
 The following notion (from \cite{DH16}) is useful in this paper:
\begin{eqnarray*}
  \ell_b(t)=\max\{k\geq1: t^k\mid(b^{O_b(t)}-1)\}.
\end{eqnarray*}
Recall that  Artin's conjecture says that if $b > 1$ is  not a perfect square, i.e.,  $b\ne n^2$ for
any $n\in\N$, then
there are infinitely many primes $p$ for which $b$ is a {\it primitive root}, i.e., $O_b(p)=p-1$. The
main results of this paper is the following:
\begin{thm}\label{thmain}
Let $(b, \D, \cC)$ be a canonical Hadamard triple. Then the following statements  hold.\\
\rm{(i)} If Artin's conjecture holds, then Conjecture \ref{conj2} is true for non-perfect square
$b>1$;\\
\rm{(ii)} If Artin's conjecture holds and $2^k\#\cC\le b=\al^{2^k}$ for some non-perfect square
$\al$, then Conjecture \ref{conj2} is true;\\
\rm{(iii)} If $b-1\not\in\cC+\cC$, then Conjecture \ref{conj2} is true;\\
\rm{(iv)} If $p^{\ell_b(p)}$ is a spectral eigenvalue for some prime $p$ satisfying
$p>(b-1)(b-\#\cC)$, then Conjecture \ref{conj2} is true.
\end{thm}

To prove our main results, we begin with the following interesting problem, which is of independent
interest: Let $0\in\cC$ be a complete representation of the group $\Z/b\Z$ and let $t\in\N$ with
$\gcd(t, b)=1$. Under what conditions can every integer \(n \in \mathbb{Z}\) be expressed in terms of
$(b, t\cC)$ using only finitely many terms? It is not difficult to get the following result:
\begin{thm}
Let $0\in\cC$ be a complete representation of the group $\Z/b\Z$ and let $t\in\N$ with $\gcd(t,
b)=1$. Then
$$\Z=\bigcup_{n=1}^\infty(t\cC+bt\cC+\cdots+b^{n-1}t\cC-b^n(T(b, t\cC)\cap\Z)).$$
\end{thm}

\section{The complete integer problem w.r.t. a pair $(b,\cC)$}
\subsection{Sub-representation of the quotient group $\Z/b\Z$}
\begin{defi}Let $b\in \Z_{\ge 2}$ and $0\in\cC\subset\Z$ be a sub-representation of the quotient group $\Z/b\Z$. We say that $s\in\Z$ is generated by $(b, \cC)$ if there exist sequences
$\{c_k\}\subset\cC$ and $\{s_k\}\subset\Z$ such that $s_0=s$ and $s_k-c_k=bs_{k+1}$ for $k\ge 0$. We denote all integers  generated by $(b, \cC)$ as $\Gamma(b, \cC)$.
\end{defi}
Clearly for each $s\in\Gamma(b, \cC)$ the corresponding sequence $\{c_k\}_{k=0}^\infty$  generated by
$(b, \cC)$ is unique, and $\Gamma(b, \cC)=\Z$ if and only if $\cC$ is a complete residue system
modulo $b$. We define
\begin{eqnarray*}
  \Lambda(b, \cC)=\bigcup_{k=1}^\infty (\cC+b\cC+\cdots+b^{k-1}\cC):=\sum_{k=1}^\infty b^{k-1}\cC.
\end{eqnarray*}
It is easy to see that $\Lambda(b, t\cC)\subset\Gamma(b, t\cC)$ for any $t\in\N$. When equality
holds, we give the following definition.
\begin{defi}
Let $t\in \N$. The $t$ is  called an complete integer w.r.t. $(b, \cC)$ if
$$\Gamma(b, t\cC)=\Lambda(b, t\cC).$$
Otherwise $t$ is called incomplete.
\end{defi}
\begin{rem}
Here we do not assume that $\gcd(t, b)=1$.
\end{rem}
\begin{prop}\label{prop2.2}
Let $b>1$ be an integer and let $0\in\cC\subset\Z$  with  all elements of it being in different
cosets of $\Z/b\Z$. Let $t\in \N$ with $\gcd(t,b)=1$. Then the following statements are equivalent:\\
\rm{(i)} $t\in\N$ is incomplete w.r.t. $(b, \cC)$;\\
\rm{(ii)} There exist $c_k\in\cC,1\leq k\leq N$ such that
\begin{eqnarray*}
  t\frac{c_1+c_2b+\cdots+c_Nb^{N-1}}{b^N-1}\in\Z\setminus\{0\},
\end{eqnarray*}
where $N$ is the minimal number such that the above holds;\\
\rm{(iii)} There exists $\{c_k\}_{k=1}^N\ne\{0\}$ such that
$$x_{k+1}=\frac{x_k+tc_k}b\in\Z\quad\text{for $1\le k<N$ and}\quad x_1=\frac{x_N+tc_N}b\in\Z.$$
\rm{(iv)} $T(b, t\cC)\cap\Z\ne\{0\}$, where
$T(b, t\cC)=\{\sum_{j=1}^\infty\frac{tc_j}{b^j}:c_j\in \cC\}$.
\end{prop}
\proof
(i)$\Leftrightarrow$(ii). For any $s\in\Gamma(b, t\cC)\setminus \Lambda(b,t\cC)$ there exist
$c_1\in\cC$ and $s_1\in\Z$ satisfying that $s-tc_1=bs_1$. Similarly one has for $n\ge 1$
\begin{eqnarray*}
  s=tc_1+btc_2+\cdots+b^{n-1}tc_{n}+b^{n}s_n,
\end{eqnarray*}
where $s_n\in \Z$ and $c_i\in\cC,1\leq i\leq n$. Then $\{s_n\}_{n=1}^\infty$ is bounded, eventually
periodic and all $s_n\ne 0$ by the assumption. That is, there exist the smallest positive integers
$\ell$ and $N$ such that
\begin{eqnarray*}
  s=tc_1+btc_2+\cdots+b^{\ell-1}tc_\ell+b^\ell s_\ell
\end{eqnarray*}
and
\begin{eqnarray*}
  s_\ell&=&tc_{\ell+1}+btc_{\ell+2}+\cdots+b^{N-1}tc_{\ell+N}+b^N s_{\ell+N}\\
  &=&tc_{\ell+1}+btc_{\ell+2}+\cdots+b^{N-1}tc_{\ell+N}+b^N s_{\ell},
\end{eqnarray*}
which implies that
\begin{eqnarray*}
  s_\ell=-\frac{tc_{\ell+1}+btc_{\ell+2}+\cdots+b^{N-1}tc_{\ell+N}}{b^N-1}\in \Z\setminus\{0\}.
\end{eqnarray*}
Hence (ii) follows. Conversely, set
\begin{equation}\label{eqqtb}
s=-t\frac{c_1+bc_2+\cdots+b^{N-1}c_N}{b^N-1}\in\Z\setminus\{0\}.
\end{equation}
Therefore
\begin{align*}
s &= tc_1+btc_2+\cdots+b^{N-1}tc_N+b^Ns \\
  &= \sum_{k=1}^{n} b^{(k-1)N} \Bigl(tc_1+btc_2+\cdots+b^{N-1}tc_N\Bigr) + b^{nN}s.
\end{align*}
This implies that $s\in \Gamma(b,t\cC)$. If $s\in\Lambda(b, t\cC)$, then there exist $c_k'\in\cC$,
$1\le k\le M$, such that $s=t\sum_{k=1}^Mb^{k-1}c_k'$. According to the definition of $\cC$, $\gcd(t,
b)=1$ and the above equality, this is impossible. Hence, (ii)$\Rightarrow$(i).

(ii)$\Leftrightarrow$(iii). From (iii) we define $c_{k+N}=c_k$  and $x_{k+1}=\frac{x_k+tc_k}b$ for
$k\ge 1$.  Then $x_{N+l}=x_l\ne 0$ for some $l$ with $1\leq l\leq N-1$. Iterating them one gets that
\begin{eqnarray*}
       x_l=x_{N+l}=\frac{x_l+tc_l+tbc_{l+1}+\cdots+tb^{N-1}c_{N+l-1}}{b^N},
\end{eqnarray*}
which is equivalent to that
\begin{eqnarray}\label{eq2.1}
       x_l=t\frac{c_l+bc_{l+1}+\cdots+b^{N-1}c_{N+l-1}}{b^N-1}\in\Z\setminus\{0\}.
\end{eqnarray}
Then (ii) follows. Conversely, let $N$ be the smallest integer such that (ii) holds.  Set
$$y_1=t\frac{c_1+c_2b+\cdots+c_Nb^{N-1}}{b^N-1}\in\Z\setminus\{0\}.$$
Define $c_{N+k}=c_k$ for $k\ge 1$ and
$$y_{k+1}=\frac{y_k+tc_k}b, \qquad \text{for $k\ge 1$.}$$
To show that the sequence $\{y_k\}_{k=1}^\infty$ is an integer cycle w.r.t. $(b, t\cC)$, it is
sufficient  to show that all $y_k\in\Z$ and $y_{N+1}=y_1$. The assertion $y_{N+1}=y_1$ follows from
 \begin{eqnarray*}
       y_{N+1}&=&\frac{y_1+tc_1+tbc_{2}+\cdots+tb^{N-1}c_{N}}{b^N}\\
       &=&\frac{t\frac{c_1+c_2b+\cdots+c_Nb^{N-1}}{b^N-1}+tc_1+tbc_{2}+\cdots+tb^{N-1}c_{N}}{b^N}=y_1.
\end{eqnarray*}
Noticing that $y_{N}=by_{N+1}-tc_N\in\Z$, after reasoning, all $y_k\in\Z$. Hence
(ii)$\Leftrightarrow$(iii).

Now we prove (ii)$\Leftrightarrow$(iv). (ii)$\Rightarrow$(iv) is trivial by the definition of $T(b,
t\cC)$. Next we prove that (iv)$\Rightarrow$(ii). Let $x_0\in T(b, t\cC)\cap(\Z\setminus\{0\})$. Then
there exists a unique code $\{c_k\}_{k=1}^\infty\subset\cC$ such that $x_0=t\sum_{k=1}^\infty
c_kb^{-k}$. Then there exist the smallest integers $\ell,N\in\N$ such that
\begin{eqnarray*}
  x_0&=&t\sum_{k=1}^\ell c_kb^{-k}+t\left(\sum_{k=\ell+1}^{\ell+N}c_kb^{-k}\right)\sum_{k=1}^\infty
  b^{-(k-1)N}\\
  &=&t\sum_{k=1}^\ell c_kb^{-k}+tb^{-\ell}\frac{\sum_{k=1}^{N}c_{\ell+k}b^{N-k}}{b^N-1},
\end{eqnarray*}
which implies that $c_\ell\ne c_{\ell+N}$ and
\[t\frac{\sum_{k=1}^{N}c_{\ell+k}b^{N-k}}{b^N-1}\in \Z.\]
If the above is zero, then  $x_0=t\sum_{k=1}^\ell c_kb^{-k}$, which is impossible because that
$\gcd(t, b)=1$, all elements of $\cC$ are in different coset of $\Z/b\Z$ and $0\in \cC$. Hence (ii)
follows.
\ez
\begin{rem}\label{nronne}
In Proposition \ref{prop2.2}, (ii) and (iii) are equivalent without the condition $\gcd(t,b)=1$.
Moreover, the condition $\gcd(t,b)=1$ is not used in (i)$\Rightarrow$(ii).
\end{rem}

\begin{rem}\label{nonne}
In Proposition \ref{prop2.2},  if we restrict $\cC$ to the nonnegative integers, then (i), (ii) and
(iii) are equivalent without the condition $\gcd(t,b)=1$. Indeed, in the proof of
(ii)$\Rightarrow$(i) in Proposition \ref{prop2.2}, by \eqref{eqqtb}, we have $s<0$ when $\cC\subset
\Z_{\geq0}$ and hence $s\notin\Lambda(b,t\cC)~(\subset\Z_{\geq0})$. Together with Remark
\ref{nronne}, we obtain the desired result.
\end{rem}

\begin{rem}
We say that a sequence $\{x_k\}_{k=1}^N$ is an integer cycle w.r.t. $(b, t\cC)$ if (iii) in
Proposition \ref{prop2.2} holds. Equivalently, a sequence $\{x_k\}_{k=1}^\infty$ is an integer cycle
w.r.t. $(b, t\cC)$ if there exists $N$ such that $\{x_k\}_{k=1}^N$ is an integer cycle w.r.t. $(b,
t\cC)$ and it is a period sequence with period $N$.
\end{rem}

\begin{coro}
 Let $t\in\N$ be an incomplete integer w.r.t. $(b, \cC)$ with $\gcd(t,b)=1$ and let $\gcd(k,b)=1$,
 then so it $kt$ for $k\in\N$. In particular, if $1$ is incomplete, then all $t\in\N$ with
 $\gcd(t,b)=1$ are incomplete.
\end{coro}
\proof
The assertion follows from (ii) of Proposition \ref{prop2.2} immediately.
\ez
\begin{coro}\label{cororgx}
If $k(b-1)\in\cC$ for some $k\in\Z\setminus\{0\}$, then any $t\in\N$ with $\gcd(t,b)=1$ is incomplete
w.r.t. $(b, \cC)$.
\end{coro}
\proof Assume $k(b-1) \in \mathcal{C}$. For any $t \in \mathbb{N}$ we have
$$t\frac{k(b-1)}{b-1}=tk\in\Z\setminus\{0\}.$$
Hence the assertion follows by (ii) of Proposition \ref{prop2.2}.\ez

By Remark \ref{nonne}, we have the following two better results when $\cC\subset \Z_{\geq0}$.
\begin{coro}
 Let $t\in\N$ be an incomplete integer w.r.t. $(b, \cC)$. If $\cC\subset \Z_{\geq0}$, then so it $kt$
 for $k\in\N$. In particular, if $1$ is incomplete, then all $\N$ are incomplete.
\end{coro}
\begin{coro}\label{corortd}
If $k(b-1)\in\cC$ for some $k\in\Z\setminus\{0\}$ and $\cC\subset \Z_{\geq0}$, then any $t\in\N$ is
incomplete w.r.t. $(b, \cC)$.
\end{coro}

\begin{prop}\label{prop2.10}
Let  all elements of $\cC$ be in different cosets of $\Z/b\Z$. If $\{x_k\}_{k=1}^\infty$ and
$\{y_k\}_{k=1}^\infty$ are two nontrivial integer cycle w.r.t. $(b, t\cC)$, then either $\{x_k: k\ge
1\}=\{y_k: k\ge 1\}$ or $\{x_k: k\ge 1\}\cap\{y_k: k\ge 1\}=\emptyset$.
\end{prop}
\proof Suppose that $x_{k+1}=\frac{x_k+tc_k}b$, $y_{k+1}=\frac{y_k+td_k}b$ and $x_m=y_n$ for some $m,
n\in\N$. Choose $N$ to be a common period of the two sequences, similar to \eqref{eq2.1}, we have
\begin{eqnarray*}
  x_m&=&t\frac{c_{m}+bc_{m+1}+\cdots+c_{m+N-1}b^{N-1}}{b^N-1}=y_n\\
  &=&t\frac{d_{n}+bd_{n+1}+\cdots+d_{n+N-1}b^{N-1}}{b^N-1},
\end{eqnarray*}
which implies $c_{m+k}=d_{m+k}$ for $k=0,1,\ldots,N-1$ and thus $x_{m+i}=y_{n+i}$ for all $i\geq1$ by
the assumption. Then $\{x_k: k\ge 1\}=\{y_k: k\ge 1\}$ by the periodicity of the two sequences. We
finish the proof.
\ez
\begin{thm}\label{theo2.9}
Let  all elements of $\cC$ be in different cosets of $\Z/b\Z$. Let $t\in \N$ with $\gcd(t,b)=1$. Then
$T(b, t\cC)\cap\Z$ is the disjoint union of all integer cycles w.r.t. $(b, t\cC)$ (including the
trivial case $\{0\}$).
\end{thm}
\proof If $t$ is complete w.r.t. $(b, \cC)$, then $T(b, t\cC)\cap\Z=\{0\}$ and the assertion follows
by Proposition \ref{prop2.2}.  If $t$ is incomplete w.r.t. $(b, \cC)$, we use the formula in the
proof of Proposition \ref{prop2.2} from (iv) to (ii), for any $x_0\in T(b, t\cC)\cap
\Z\setminus\{0\}$,  there exist the smallest integers $\ell\ge 0$ and $N\ge 1$ such that
\begin{eqnarray*}
  x_0=t\sum_{k=1}^\ell c_kb^{-k}+tb^{-\ell}\frac{\sum_{k=1}^{N}c_{\ell+k}b^{N-k}}{b^N-1}
\end{eqnarray*}
and $c_\ell\ne c_{\ell+N}$
and
$$\alpha:=t\frac{\sum_{k=1}^{N}c_{\ell+k}b^{N-k}}{b^N-1}\in\Z\setminus\{0\}.$$
  Clearly we can assume that $\gcd(\alpha, t)=1$. Then $b^N\equiv 1\pmod t$ and thus $O_b(t)\mid N$.
  If $\ell\ge 1$, further we have
\begin{eqnarray*}
  \frac{t\frac{\sum_{k=1}^{N}c_{\ell+k}b^{N-k}}{b^N-1}}b+\frac{tc_\ell}b=b^{\ell-1}x_0-t\sum_{k=1}^{\ell-1}
  c_kb^{\ell-1-k},
\end{eqnarray*}
which is equivalent to that
\begin{eqnarray*}
  \frac{\sum_{k=1}^{N}c_{\ell+k}b^{N-k}+c_\ell(b^N-1)}b=b^{\ell-1}x_0\frac{b^N-1}t-(b^N-1)\sum_{k=1}^{\ell-1}
  c_kb^{\ell-1-k}\in\Z.
\end{eqnarray*}
This forces $c_{\ell+N}=c_\ell\pmod b$. Thus $c_{\ell+N}=c_\ell$, which yields a contradiction. Hence
$\ell=0$ and $x_0$ is a cycle point w.r.t. $(b, t\cC)$. The union is disjoint by Proposition
\ref{prop2.10}. The converse is trivial.\ez
\begin{rem}
 To guarantee that $t\cC$ is a sub-representation of the  $\Z/b\Z$, we always assume that $\gcd(t,
 b)=1$. And in order to avoid the trivial case, we always assume that $k(b-1)\not\in\cC$ for any
 $k\in \Z\setminus\{0\}$.
\end{rem}
\begin{prop}\label{prop2.9}
Let $t\in\N$ be a complete number w.r.t. $(b, \cC)$ with $\gcd(t,b)=1$. If $p_1, p_2, \ldots, p_k$
are different prime factors of $b$, then $tp_1^{\alpha_1}p_2^{\alpha_2}\cdots{}\allowbreak p_k^{\alpha_k}$ is
complete for each $(\alpha_1, \ldots, \alpha_k)\in\Z_{\geq0}^k$.
\end{prop}
\proof If $tp_1^{\alpha_1}p_2^{\alpha_2}\cdots p_k^{\alpha_k}$ is incomplete, by (ii) of Proposition
\ref{prop2.2}, $t$ is incomplete, which contradicts the assumption of $t$. We finish the proof. \ez
\begin{defi}
Let $t\in\N$ and $1$ is complete w.r.t. $(b, \cC)$. The $t$ is called a primitively incomplete
integer w.r.t. $(b, \cC)$ if $t$ is incomplete and any proper factor $t'\mid t, t'\ne t$, is
complete.
\end{defi}
\begin{prop}\label{prop2.12}
Let $t\in\N$ be a primitive incomplete integer w.r.t.  $(b, \cC)$. Then the following statements
hold:\\
  (1) $\gcd(t, b)=1$;\\
  (2) Let $\{x_k\}_{k=1}^\infty$ be an integer cycle w.r.t. $(b, t\cC)$. Then $\gcd(t, x_k)=1$ for
  all $k$;\\
  (3) Suppose that $T(b, \cC-\cC)\cap \Z=\{0\}$. Let $\{x_k\}_{k=1}^\infty$ be an integer cycle
  w.r.t. $(b, t\cC)$. Then the smallest period of it is $O_b(t)$.
\end{prop}
\proof (1) follows simply from (ii) of Proposition \ref{prop2.2} and $\gcd(d, b^{n}-1)=1$ for any
$n\geq 1$ if $d\mid b$. And (2) follows by (ii) of Proposition \ref{prop2.2} and \eqref{eq2.1}.

Now we prove (3). Let $\{x_k\}_{k=1}^\infty$ be an integer cycle w.r.t. $(b, t\cC)$ with the smallest
period $N$. By the proof of (2) we have $O_b(t)\mid N$. From $x_2=\frac{x_1+tc_1}b$ we have $x_2=b^{O_b(t)-1}x_1\pmod t$. Then $x_{O_b(t)+1}=x_1\pmod t$.
Notice that $x_{O_b(t)+1}, x_1\in T(b, t\cC)$. Then there exists $\alpha\in\Z$ such that
$$\alpha=\frac{x_{O_b(t)+1}-x_1}t\in T(b, \cC-\cC).$$
This implies that $x_{O_b(t)+1}=x_1$ and (3) follows.\ez

\begin{rem} We guess that the assumption $T(b, \cC-\cC)\cap \Z=\{0\}$ in (3) of
Proposition \ref{prop2.12} is redundant.
\end{rem}

According to Proposition \ref{prop2.12}, we have the following definition.
\begin{defi}
  $t\in\N$ is called primitively complete w.r.t. $(b, \cC)$ if it is complete and $\gcd(t, b)=1$.
\end{defi}

\begin{thm}\label{theonotd}
Let $b\in \Z_{\geq2}$ and $0\in\cC\subset\Z$ be a sub-representation of the quotient group $\Z/b\Z$.
Suppose $t=1$ is complete w.r.t. $(b, \cC)$. If $\gcd(t,b)=1$ or $\cC\subset \Z_{\geq0}$, then there
are infinitely many primitively incomplete integers w.r.t. $(b, \cC)$.
\end{thm}
\proof
Suppose that there are only finitely many primitively incomplete integers w.r.t. $(b, \cC)$. Write
all of them, which are not factors of  $b-1$, as $\A=\{t_1, t_2, \ldots, t_m\}$, i.e., $t_k\nmid
(b-1)$ for $1\le k\le m$. Let $\alpha\in\N+1$ such that
$q:=O_b((b-1)^\alpha t_1 t_2\cdots t_m)>b-3$. The existence of such an $\alpha$ is guaranteed by the
fact that  $\limsup\limits_{n\to\infty}O_b(n)=\infty$.  Denote \[t=\frac{b^{q+1}-1}{b-1}\in\N.\]
By the definition of $q$, there exists $\beta\in\N$ such that $b^q=1+\beta (b-1)^\alpha t_1 t_2\cdots
t_m$ and
\begin{eqnarray*}
  (b-1)t=b^{q+1}-1=b-1+b\beta (b-1)^\alpha t_1 t_2\cdots t_m.
\end{eqnarray*}
 Then $t=1+b(b-1)^{\alpha-1}\beta t_1 t_2\cdots t_m$.

As $q>b-3$, clearly for any nonzero $c\in\cC$
\[tc\frac{1+b+\cdots+b^{b-2}}{b^{q+1}-1}=c\frac{1+b+\cdots+b^{b-2}}{b-1}\in \Z\setminus\{0\}.\]
Then by  (ii) of Proposition \ref{prop2.2} and Remark \ref{nonne}, if $\gcd(t,b)=1$ or $\cC\subset
\Z_{\geq0}$,  the $t$ is  incomplete. Hence there exists $t'\mid t$ such that $t'>1$ is  primitively
incomplete.
If $t'\mid (b-1)$ or $t'=t_k$ for some $1\le k\le m$, then it contradicts the fact that
$t=1+b(b-1)^{\alpha-1}\beta t_1 t_2\cdots t_m$.  Hence $t'$ is a new primitively incomplete number
and thus the assertion follows.\ez

\begin{rem}
Under the above conditions in Theorem \ref{theonotd}, it is not clear whether there exist infinitely
many primitively complete integers w.r.t. $(b, \cC)$.
\end{rem}

\begin{exam}
Let $b=4$ and $\cC=\{0,2\}$. Note that $3\cdot\frac{2}{4-1}=2\in\mathbb{Z}\setminus\{0\}$. Then by
(ii) of Proposition \ref{prop2.2}, $t=3$ is incomplete with respect to $(4,\{0,2\})$. Next, we prove
that
for each $n\in\N$, $5^n$ is complete with respect to $(4,\{0,2\})$. The idea of the proof comes from
\cite{DH16}.

Suppose $5^n$ is incomplete. Write $t=5^n$. Since \(\gcd(4,t) = 1\), by Proposition \ref{prop2.2}
(iii) there exists a non-trivial integer cycle $\{x_k\}_{k=1}^N$ with respect to $(4,t\{0,2\})$ such
that
\begin{equation}\label{eqeqe}
x_{k+1}=\frac{x_k+tc_k}4\in\Z\quad\text{for $1\le k<N$ and}\quad x_1=\frac{x_N+tc_N}4\in\Z,
\end{equation}
where all $c_i\in\{0,2\}$.  By \eqref{eq2.1}, all $x_i\in (0,\frac{2t}{3})$. By \eqref{eqeqe}, we
have $x_k+tc_k\equiv 0\pmod{4}$ for all $1\leq k\leq N$, which implies that all $x_k$ are even.

Write $y_i = \frac{x_i}{2}$. Then  $\{y_i\}$ forms an integer cycle w.r.t. $\{4,t\{0,1\}\}$ with all
$y_i \in (0,\frac{t}{3})$. Note that
\[
4^{(5^{N-1})} \equiv -1 \pmod{t}.
\]

Take $k = 5^{N-1}$. Then for any $y$ in the cycle, we have $4^k y \equiv -y \pmod{t}$. And by the
definition of the integer cycle, there exists a cycle point $y' \equiv 4^k y \pmod{t}$, which means
$y' \equiv -y \pmod{t}$. Then $y'=t-y$. But $y \leq t/3$, so $t - y > 2t/3$, contradicting $y' \leq
t/3$. Hence, we complete the proof.
\end{exam}

\subsection{Complete representation of the quotient group $\Z/b\Z$}

Given \(b \in \Z_{\geq2}\), let \(\mathcal{C} \subset \mathbb{Z}\) be a complete representation of
the quotient group $\Z/b\Z$ with \(0 \in \mathcal{C}\). For a positive integer \(t\) satisfying
\(\gcd(t, b)=1\), consider the set
\[
\Lambda(b, t\mathcal{C}) := t \bigcup_{n=1}^{\infty} \left(\mathcal{C} + b\mathcal{C} + \cdots +
b^{n-1}\mathcal{C}\right).
\]
According to Definition 2.1, the corresponding set \(\Gamma(b, t\mathcal{C})\) equals \(\mathbb{Z}\)
under this condition. At a first glimpse, a natural question arises: when does the following equality
also hold
\[
\mathbb{Z} = \Lambda(b, t\mathcal{C})?
\]
It is easy to see that the above equality cannot hold for any \(t \ge 2\). If $\Z=\Lambda(b, \cC)$,
then for any $s\in\Z$ there exist $c_1,c_2,\ldots,c_\ell$ such that
$$ s=\sum_{k=1}^\ell c_kb^{k-1}.$$
Hence, characterizing $\Z=\Lambda(b, \cC)$ may be useful for the general theory. The following
proposition, which parallels Proposition \ref{prop2.2}, provides such a characterisation. The proof
is similar and therefore omitted.

\begin{prop}\label{prop2.2add}
Let $0\in\cC$ be a complete representation of the group $\Z/b\Z$ and let $\gcd(t,b)=1$. Then the
following statements are equivalent:\\
\rm{(i)} $\Z\neq \Lambda(b, t\mathcal{C})$;\\
\rm{(ii)} there exists $c_k\in\cC$ such that
$$t\frac{c_1+bc_2+\cdots+b^{N-1}c_N}{b^N-1}\in \Z\setminus\{0\};$$
\rm{(iii)}$T(b, t\cC)\cap\Z\ne\{0\}$.
\end{prop}
\begin{rem}
In fact, Proposition \ref{prop2.2add} gives an independent proof to the following fact:
\[T(b, t\cC)\cap\Z\ne\{0\}\] for all $t\geq2$ with $\gcd(t,b)=1$.
\end{rem}
$\N$ are incomplete.
\begin{coro}\label{cotri}
If $\cC\subset\{-b+2, -b+3, \ldots, b-2\}$, then  $\Z=\Lambda(b, \cC)$.
\end{coro}
\proof Notice that $T(b, \cC)\subset(-1, 1)$. Then by (iii) of Proposition \ref{prop2.2add} the
assertion follows. \ez
\begin{coro}\label{chaad}
If $k(b-1)\in\cC$ for some $k\in\Z\setminus\{0\}$, then $\Z\ne\Lambda(b, \cC)$.
\end{coro}
\begin{exam}
Let $b=2$ and $\cC=\{0,2k+1\}$ for $k\in \Z$. Then by Corollary \ref{chaad}, it is immediate that
$\Z\ne\Lambda(2, \{0, 2k+1\})$. When $b=3$, the situation becomes very complicated. Write
$\cC=\{0,3k_1+1,3k_2+2\}$ for $k_1,k_2\in \Z$. By Corollary \ref{chaad}, we have that $3k_1+1,3k_2+2$
are both odd. When $k_1=0$ and $k_2=-1$, $\cC=\{0,1,-1\}$. By Corollary \ref{cotri},
$\Z=\Lambda(3,\{0,1,-1\})$. In fact, we can obtain a general result. When $k_1=0$ and $k_2=-3^{k-1}$
for $k\geq1$, $\cC=\{0,1,-3^k+2\}$. Next, we prove $\Z=\Lambda(3,\{0,1,-3^k+2\})$.

By Proposition \ref{prop2.2add}, it suffices to prove that $T(3,\cC)\cap\mathbb{Z}=\{0\}$. Suppose,
for contradiction, that there exists a nonzero integer $x\in T(3,\cC)\cap\mathbb{Z}$.
Then there is a sequence $\{c_j\}_{j=1}^\infty\subset \cC$ such that
\[
x = \sum_{j=1}^\infty c_j 3^{-j}.
\]

Write each $c_j$ as $c_j = a_j - 3^k b_j$, where
\[
(a_j,b_j)=\begin{cases}
(0,0) & \text{if }c_j=0,\\[2pt]
(1,0) & \text{if }c_j=1,\\[2pt]
(2,1) & \text{if }c_j=-3^k+2.
\end{cases}
\]
Then
\[
x = \sum_{j=1}^\infty a_j 3^{-j} - 3^k\sum_{j=1}^\infty b_j 3^{-j} =:u - 3^k v,
\]
where $u=\sum_{j=1}^\infty a_j 3^{-j}\in[0,1]$ and $v=\sum_{j=1}^{\infty} b_j 3^{-j}\in[0,\frac12]$.

Decompose $3^k v$ as follows:
\[
3^k v = \sum_{j=1}^{k} b_j 3^{k-j} + \sum_{j=1}^{\infty} b_{k+j} 3^{-j} =: N + w,
\]
where $N=\sum_{j=1}^{k} b_j 3^{k-j}\in\mathbb{Z}_{\ge0}$ and $w=\sum_{j=1}^{\infty} b_{k+j}
3^{-m}\in[0,\frac12]$.
Thus
\[
x = u - (N+w) = (u-w) - N.
\]
Then $u-w$ is an integer, which implies that  $u-w$ can only be $0$ or $1$. If $u-w=1$, then
$u=w+1\le1$,
which together with $w\ge0$ forces $u=1$ and $w=0$. By $u=1$, we have that all $a_j=2$, which implies
that all $c_j=-3^k+2$. Then $x=\frac{-3^k+2}{2}\notin \Z$. Contradiction. Hence we must have $u-w=0$,
and consequently $x=-N<0$.

Now $u$ and $w$ are two ternary expansions that are equal. Note that equality of the two ternary
numbers can occur only in the following two ways:

(i) The standard expansion of $u$ coincides with that of $w$, so $a_j = b_{k+j}$ for all $j\ge1$ and
consequently $a_j\in\{0,1\}$.

(ii) There exists an index $m_0\ge1$ such that $w$ has a finite expansion ending with a $1$ followed
by all $0$'s, i.e., $b_{k+m_0}=1$ and $b_{k+m}=0$ for all $m>m_0$. In this case $w$ can also be
represented as $0.b_{k+1}\cdots b_{k+m_0-1}0\,222\ldots$; then $a_j=b_{k+j}$ for $j<m_0$, $a_{m_0}=0$
(while $b_{k+m_0}=1$), and $a_j=2$ for all $j>m_0$.

We next show that cases (i) and (ii) both cannot happen under our assumptions. Assume case (i) holds.
Then $a_j = b_{k+j}$ for all $j\ge1$, and $a_j\in\{0,1\}$.
Now for each $j\ge1$,
\[
c_j = a_j - 3^k b_j = b_{k+j} - 3^k b_j.
\]
If $b_j=0$, then $c_j = b_{k+j}\in\{0,1\}$, which is allowed. If $b_j=1$, then
$c_j = b_{k+j} - 3^k$ equals either $-3^k$ (if $b_{k+j}=0$) or $1-3^k$ (if $b_{k+j}=1$);
neither of these belongs to $\cC$. Thus we must have $b_j=0$ for all $j$.
Consequently $a_j = b_{k+j}=0$ for all $j$, so every $c_j=0$, and $x=0$. Contradiction.

Assume case (ii) holds. Then for every $j>m_0$ we have $a_j=2$ and $b_{k+j}=0$.
For such $j$, consider $c_j = a_j - 3^k b_j = 2 - 3^k b_j$. For $c_j\in \cC=\{0,1,-3^k+2\}$, we must
have $b_j=1$ (so that $c_j=2-3^k$).
Thus $b_j=1$ for all $j>m_0$. But then for $j=k+m$ with $m>m_0$, we have
$b_{k+m}=0$ while $b_{k+m}=1$, a contradiction.

Hence, we complete the proof.
\end{exam}
\begin{thm}
Let $0\in\cC$ be a complete representation of the group $\Z/b\Z$ and let $t\in\N$ with $\gcd(t,
b)=1$. Then
$$\Z=\bigcup_{n=1}^\infty(t\cC+bt\cC+\cdots+b^{n-1}t\cC-b^n(T(b, t\cC)\cap\Z)).$$
\end{thm}
\proof The assertion follows by the  proof  of the (i)$\Rightarrow$(ii) of Proposition
\ref{prop2.2add}, where we need to see that $s_\ell\in -T(b, t\cC)$ is trivial.

\section{Complete numbers}

Let $b>2$ be an integer. The following two results are well known (see, e.g., \cite{LW02}) and easy
to check directly:
\begin{prop}\label{prop3.1}
Let $(b, \D, \cC)$ be an integral Hadamard triple. Then\\
(i) Both $\D$ and $\cC$ are sub co-sets of modulo $b$ in $\Z$;\\
(ii) For any $\D^*$ and $\cC^*$ satisfying $\D^*\equiv\D\pmod b$ and $\cC^*\equiv\cC\pmod b$, $(b,
\D^*, \cC^*)$ is also an integral Hadamard triple;\\
(iii) For any $c, d\in \R$, $(b, \D+c, \cC+d)$ is  a Hadamard triple.
\end{prop}
\begin{rem}If $b=\#\D$, then both $\D$ and $\cC$ are the complete representations of $\Z/b\Z$, and
$\mu_{b, \D}$ is the Lebesgue measure restricted on the tile $T(b, \D)$. In this case, the spectrum
of it is finite up to translations by An and the first author \cite{AH26}, which has no complete
number problem on it.
\end{rem}
\begin{thm}[\cite{LW02}]\label{theo3.2}
Let $t\in\N$ and let $(b, \D, t\cC)$ be an integral Hadamard triple with $\gcd(\D)=1$ and
$0\in\D\cap\cC$. Then   $(\mu_{b, \D}, t\La(b, \cC))$ is not a spectral pair if and only if $t$ is
incomplete w.r.t. $(b, \cC)$.
\end{thm}
\begin{coro}
If $b-1\in \cC$, then $(\mu_{b, \D}, t\La(b, \cC))$ is not a spectral pair for each $t\in \N$ with
$\gcd(t,b)=1$.
\end{coro}
\proof By Corollary \ref{cororgx}, any $t\in\N$ is incomplete w.r.t. $(b, \cC)$ if $b-1\in \cC$.
Hence the assertion follows by Theorem \ref{theo3.2}.\ez

The following result is due to Dutkay and Jorgensen \cite{DJ06}.
\begin{thm}
Let $t\in\N$ and let $(b, \D, t\cC)$ be an integral Hadamard triple. Then $\mu_{b, \D}$ is a spectral
measure with a spectrum $\Gamma(b, t\cC)$. Moreover,
$$\Gamma(b, \cC)=\bigcup_{k=1}^\infty(t\cC+bt\cC+\cdots+b^{k-1}t\cC-b^k(T(b, t\cC)\cap\Z)).$$
\end{thm}

\begin{prop}\label{theo3.6}
Let $(b, \D, \cC)$ be an integral Hadamard triple with $\gcd(\D)=1$, $0\in\D\cap\cC$ and
$\cC\subset\{0, 1, \ldots, b-1\}$. Then
$\La(b, \cC)$ is a spectrum of $\mu_{b, \D}$ if and only if $b-1\not\in\cC$.
\end{prop}
\proof The necessity follows by Corollary \ref{corortd} and Theorem \ref{theo3.2}. Conversely, if
$b-1\not\in\cC$, then $T(b, \cC)\subset \left[0, \frac{b-2}{b-1}\right]$. By (iv) of Proposition
\ref{prop2.2} there is no integral cycle w.r.t $(b, \cC)$, so the sufficiency follows.\ez

\begin{thm}
Let $(b, \D)$ be an admissible pair with $\gcd(\D)=1$, $0\in \D$ and $b>\#\D$. Then there exists a
digit set $\cC$ with $0\in\cC\subset\{0, 1, \ldots, b-2\}$ such that the $(\mu_{b, \D}, \Lambda(b,
\cC)$ is a spectral pair.
\end{thm}
\proof Let $\E$ such that $(b, \D, \E)$ is an integral Hadamard triple with $0\in\E$. If we appoint
that $a\pmod b\subset \{0, 1, \ldots, b-1\}$, then $(b, \D, \E\pmod b)$ is an integral Hadamard
triple too. Denote $\E\pmod b=\{e_1=0, e_2, \ldots, e_q\}$ in the increasing order. If $\E\pmod
b\subset\{0, 1, \ldots, b-2\}$,  take $C=\E\pmod b$ and the assertion follows by Proposition
\ref{theo3.6};   If $b-1\in\E\pmod b$, there exists $i$ such that $e_i+1<e_{i+1}$. Write
$\cC=(\E\pmod b-e_{i+1})\pmod b$. We claim that $\cC\subset\{0, 1, \ldots, b-2\}$ is the desired set.
In fact,
$\cC=\{e_1-e_{i+1}+b,e_2-e_{i+1}+b,\ldots,e_i-e_{i+1}+b,0,e_{i+2}-e_{i+1},\ldots,e_q-e_{i+1}\}$. The
maximal element of $\cC$ is $\max\{e_{i}-e_{i+1}+b, e_q-e_{i+1}\}$, which is not equal to $b-1$. This
implies the claim by Proposition \ref{theo3.6} again.
 \ez

According to Section 2 and Theorem \ref{theo3.2}, to consider the complete number problem we can
assume that $\gcd(t, b)=1$ without loss of generality.

To prove Theorem \ref{thmain}, the key is the following theorem:
\begin{thm}\label{theobasic}
Let $t\in\N$ with $\gcd(t, b)=1$ and let $(b, \D, \cC)$ be a canonical Hadamard triple. Then the
following statements are equivalent:\\
\rm{(i)} The  $(\mu_{b, \D}, t\La(b, \cC))$ is not a spectral pair;\\
\rm{(ii)} $t$ is an incomplete number w.r.t. $(b, \cC)$;\\
\rm{(iii)}  $T(b, t\cC)\cap\N\ne\emptyset$;\\
\rm{(iv)} There exist $c_k\in\cC$, $1\le k\le n$, such that
\begin{eqnarray*}
  t\frac{c_1+c_2b+\cdots+c_{n}b^{n-1}}{b^{n}-1}\in tT(b, \cC)\cap\N,
\end{eqnarray*}
where $O_b(t)$ is the minimal $n$ such that the above holds;\\
\rm{(iv)} There exists $x\in \N$ such that
\[\left\{b^kx\pmod t\right\}_{k=1}^{O_b(t)}\subset T(b, t\cC)\cap\N,\]
where the elements in the left are pairwise different.
\end{thm}

Let $\gcd(t, b)=1$. Recall that the order of $b$ modula  $t$ is defined by
\begin{eqnarray*}
  O_b(t)=\min\{k: b^k\equiv1\pmod t\}.
\end{eqnarray*}
Note that $O_b(t)$ is a factor of the Euler function $\psi(t)$. In this paper we stipulate for
$p\in\Z$
\begin{eqnarray*}
  p\pmod t\in \{0, 1, \ldots, t-1\}.
\end{eqnarray*}
\begin{thm}\label{theo2.7}
Let $t\in\N$ with $\gcd(t, b)=1$ and $0\in\cC\subset\{0, 1, \ldots, b-2\}$. Then  there is a
non-trivial integral cycle w.r.t. $(b, t\cC)$ if and only if there exists $x\in \N$  such that
  $$\left\{b^{O_b(t)-k}x\pmod t\right\}_{k=1}^{O_b(t)}\subset T(b, t\cC)\cap\N.$$
\end{thm}
\proof We firstly show the necessity. By the definition of a non-trivial integral cycle
$\{x_k\}_{k=1}^\infty$ w.r.t. $(b, t\cC)$  and Proposition \ref{prop2.2}, we have $x_1\in tT(b,
\cC)\cap\N$ and for $1\le k\le O_b(t)$
\begin{eqnarray*}
  b^{k}x_{k+1}=b^{k-1}x_k+tb^{k-1}c_k=\cdots=x_1+tc_1+tbc_2+\cdots+tb^{k-1}c_k.
\end{eqnarray*}
Then
\[
x_{k+1} \equiv b^{O_b(t)-k}x_1 \pmod{t}.
\]
Moreover, since all
\begin{eqnarray}\label{add2.6}
  x_m\in T(b, t\cC)\subset \bigg[0, t\frac{b-2}{b-1}\bigg]\subset [0,t),
\end{eqnarray}
we actually have $x_{k+1} = b^{O_b(t)-k}x_1 \pmod{t}$. Taking $x = x_1$, we conclude that
\begin{eqnarray*}
\left\{b^{O_b(t)-k}x \pmod{t}\right\}_{k=1}^{O_b(t)} = \{x_2, x_3, \ldots, x_{O_b(t)+1}\} \subset
T(b, tC) \cap \N
\end{eqnarray*}
and the necessity follows.

Conversely, set $x_k=b^{O_b(t)-k}x\pmod t$ for $1\le k\le O_b(t)$ and $x_{k+O_b(t)}=x_{k}$ for $k\ge
1$. Since
$ bx_{k+1}=x_k\pmod t$ for $1\le k< O_b(t)$, there exists $d_k\in\Z$ such that $bx_{k+1}=x_k+td_k$.
By $x_{k+1}, x_k\in T(b, t\cC)$, there exist $c_j\in\cC$ such that $x_{k+1}=\sum_{j=1}^\infty tc_j
b^{-j}$. Then $bx_{k+1}-tc_1\in T(b, t\cC)$. Thus by \eqref{add2.6}
\begin{eqnarray*}
  td_k=tc_1+bx_{k+1}-tc_1-x_k\in tc_1+tT(b, \pm\cC)\subset tc_1+t(-1, 1).
\end{eqnarray*}
This forces $d_k=c_1\in\cC$. To get the proof we need to show that
 $$x_{O_b(t)+1}=\frac{x_{O_b(t)}+td_{O_b(t)}}b=x_1.$$
 In fact, by definition we have $x_{O_b(t)}=x\pmod t$ and $x_1=b^{O_b(t)-1}x\pmod t$. Then $
 bx_{1}\equiv x_{O_b(t)}\pmod t$. Repeating the above proof the desired result follows.
 Hence, the sequence $\{x_k\}_{k=1}^\infty$ is a non-trivial integral cycle w.r.t. $(b, t\cC)$.
\ez

\noindent{\bf Proof of Theorem \ref{theobasic}.}
 By the results in Section 2, we only need to show the minimality of $O_b(t)$ in both (ii) and (iii)
 and the pairwise distinctness in (iv). However, these three assertions are equivalent by Proposition
 \ref{prop2.2} and Theorem \ref{theo2.7}. Now we prove the pairwise distinctness in (iv). If there
 exist $1\le i<j\le O_b(t)$ such that $b^jx=b^ix\pmod t$, that is
$b^jx=b^ix+tm$ for some $m\in\Z$. By Proposition \ref{prop2.12} we have $x\mid m$ and thus
$b^j=b^i\pmod t$, which contradicts the definition of $O_b(t)$. Hence we finish the proof.
\ez

\section{The proofs of main results}
\begin{thm}\label{thmpii}
Let $(b, \D, \cC)$ be a canonical Hadamard triple. There are infinitely many primitively non spectral
eigenvalues of the spectral pair $(\mu_{b, \D}, {}\allowbreak\La(b, \cC))$.
\end{thm}
\proof
It is a direct corollary of Theorem \ref{theonotd} and Theorem \ref{theobasic}.
\ez

We now turn to the complementary case of Theorem \ref{thmpii}, which remains open and leads to the
following conjecture:
\begin{conj}\label{conj4.2}
Let \((b,\mathcal{D},\cC)\) be a canonical Hadamard triple. Then there are infinitely many
primitively spectral integer eigenvalues of the spectral pair
\((\mu_{b,\mathcal{D}},\Lambda(b,\cC))\).
\end{conj}

In the rest of this section, we present several results related to this conjecture.

Let $t\in\N$ with $\gcd(t, b)=1$. As usual write $\Z/t\Z=\{[0]_t, [1]_t, \ldots, [t-1]_t\}$, where
$[x]_t$ denotes the equivalence class of $x$ modulo $t$. For $x\in\N$  we denote
\begin{eqnarray*}
  G_b(t, x)=\left\{[bx]_t, [b^2x]_t, \ldots, [b^{O_b(t)}x]_t\right\}\subset \Z/t\Z.
\end{eqnarray*}
and simply denote $G_b(t)=G_b(t, 1)$.

\begin{lem} \label{lem4.3}
Let $(b, \D, \cC)$ be a canonical Hadamard triple with  $b-1\not\in \cC+\cC$ and let $t\in\N$ with
$\gcd(t, b)=1$. If $[-1]_t\in G_b(t)$, then $t$ is a spectral eigenvalue of the spectral pair
$(\mu_{b, \D}, \La(b, \cC))$.
\end{lem}
\proof If $t\in \N$ is incomplete, by Theorem \ref{theo2.7} and Theorem \ref{theo3.2}, then there
exists $x\in T(b, t\cC)\cap\N$ such that
\[\left\{b^kx\pmod t\right\}_{k=1}^{O_b(t)}\subset T(b, t\cC)\cap\N.\]

Since $[-1]_t\in G_b(t)$, it follows that there exists $n$, $1\le n\leq O_b(t)$,  such that
$b^n\equiv-1\pmod t$. Denote $x^*=b^nx\pmod t$. Then
$x^*=t-x$ because both $x$ and $x^*$ in $T(b, t\cC)\subset [0, t\frac{b-2}{b-1}]$.

According to Proposition \ref{prop2.2} and Theorem \ref{theo2.7}, one has
\begin{eqnarray*}
  x=t\frac{c_1+bc_2+\cdots+b^{O_b(t)-1}c_{O_b(t)}}{b^{O_b(t)}-1}, \qquad
  x^*=t\frac{c_1^*+bc_2^*+\cdots+ b^{O_b(t)-1}c^*_{O_b(t)}}{b^{O_b(t)}-1},
\end{eqnarray*}
where all $c_k$ and $c_k^*$ lie in $\cC~(\subset\{0, 1, \ldots, b-2\})$. Then
\begin{eqnarray*}
 (c_1+ c_1^*)+b(c_2+c_2^*)+\cdots+b^{O_b(t)-1}(c_{O_b(t)}+c^*_{O_b(t)})=b^{O_b(t)}-1
\end{eqnarray*}
This yields $c_1+ c_1^*\equiv-1\pmod b$. Since $0\le c_k, c_k^*\le b-2$ for all $c_k\in\cC$, one has
$c_1+ c_1^*=b-1$, which contradicts the assumption $b-1\notin \cC+\cC$. Hence the result follows. \ez

\begin{lem}\label{lem4.4}
Let $(b, \D, \cC)$ be a canonical Hadamard triple with  $b-1\not\in \cC+\cC$. If $O_b(t)$ is even and
$t>2$ is a prime, then $t$ is a spectral eigenvalue of the spectral pair $(\mu_{b, \D}, \La(b, \cC)$.
\end{lem}
\proof By $b^{O_b(t)}\equiv 1\pmod t$ and $t$ is a prime, we have $b^{O_b(t)/2}\equiv-1\pmod t$ or
$b^{O_b(t)/2}\equiv1\pmod t$. By the definition of $O_b(t)$ we have $b^{O_b(t)/2}\equiv-1\pmod t$,
that is, $[-1]_t\in G_b(t)$. Then the result follows by Lemma \ref{lem4.3}.\ez

\begin{thm}
  Conjecture \ref{conj4.2} is true if $b-1\not\in \cC+\cC$.
\end{thm}
\proof
By Lemma \ref{lem4.4}, it suffices to show that there are infinitely many primes $t$ such that
$O_b(t)$ is even. We divide the proof into the following two cases.

Case 1: $b$ is not a perfect square. By Theorem 3 in Section 5.2 of \cite{IR90}, there are infinitely
many primes $t>2$ with $\gcd(t,b)=1$ such that $b$ is quadratic nonresidue modulo $t$. For any such
prime $t$, Euler's criterion gives
\begin{equation}\label{eqcocon}
b^{\frac{t-1}{2}}\equiv-1 \pmod{t}.
\end{equation}

Let $d = O_b(t)$. Then $d \mid (t-1)$. Suppose for contradiction that $d$ is odd. Since $d \mid
(t-1)$, write $t-1 = d \cdot k$ for some integer $k$, which fores $k$ to be even. Then
\[
b^{\frac{t-1}{2}} = (b^d)^{\frac{k}{2}} \equiv 1^{\frac{k}{2}} = 1 \pmod{t}.
\]
This contradicts \eqref{eqcocon}. Hence, $d$ must be even. Thus, $O_b(t)$ is even for infinitely many
primes $t$ in this case.

Case 2: $b$ is a perfect square. Write $b = a^2$ for some integer $a > 1$. Consider the sequence of
numbers
\[
a^{(2^{k+1})} + 1, \quad k = 1, 2, 3, \dots
\]
For each $k$, let $p_k$ be an odd prime divisor of $a^{2^{k+1}} + 1$ (such a prime exists since \(a >
1\), \(a^{(2^{k+1})} + 1 > 2\) and it is not a power of \(2\)). Then we have
\[
a^{(2^{k+1})} \equiv -1 \pmod{p_k},
\]
which implies
\[
b^{(2^k)} = a^{(2^{k+1})} \equiv -1 \pmod{p_k}.
\]
Consequently, $b^{(2^{k+1})} \equiv 1 \pmod{p_k}$ and $b^{(2^k)} \not\equiv 1 \pmod{p_k}$, so
$O_b(p_k)$ is even.

We claim that the primes $p_k$ are distinct. If $p_k = p_l=:p$ for some $k < l$, then we would have
simultaneously
\[
a^{(2^{k+1})} \equiv -1 \pmod{p}, \quad a^{(2^{l+1})} \equiv -1 \pmod{p}.
\]
Raising the first congruence to the power $2^{l-k}$ gives $a^{(2^{l+1})} \equiv (-1)^{(2^{l-k})} = 1
\pmod{p}$, contradicting the second congruence. Hence all $p_k$ are distinct. Thus there exist
infinitely many primes $p$ such that $O_b(p)$ is even.

Hence, we complete the proof.
\ez

For any real number $x$, we define $\lfloor x \rfloor= \max\{n \in \mathbb{Z} : n \leq x\}$ and
$\lceil x \rceil= \min\{n \in \mathbb{Z} : n \geq x\}$.
\begin{prop}\label{prop4.6}
If $O_b(t)>l_0(t)\#\cC$, where
\[l_0(t)=\left\lceil \frac{t(b-2)}{b(b-1)}\right\rceil,\] then $t$ is not a primitive non spectral
eigenvalue of $(\mu_{b, \D}, \La(b, \cC))$.
\end{prop}
\proof If $t$ is  primitive non spectral eigenvalue, then there is an integer cycle
$\{x_k\}_{k=1}^{O_b(t)}$ w.r.t. $(b, t\cC)$ and all $x_k$ are pairwise different by Theorem
\ref{theobasic}. By $bx_{k+1}=x_k+tc_k\equiv0\pmod b$ for $k\ge 1$, all $x_k\in -t\cC\pmod b$. Then
\[\#\{x_k: x_k\in [0, b), 1\le k\le O_b(t)\}\le\#\cC\]
and in general
\begin{eqnarray}\label{eq4.1}
  \#\{x_k: x_k\in [0, lb), 1\le k\le O_b(t)\}\le l\#\cC.
\end{eqnarray}
Note that $\gcd(b-1, b-2)=1$ and $(b-1)\nmid t$. There exists $m$ such that
\[(m-1)b\le \bigg\lfloor t\frac{b-2}{b-1}\bigg\rfloor<t\frac{b-2}{b-1}<m b.\]
Therefore
$$m=\left\lceil \frac{t(b-2)}{b(b-1)}\right\rceil:=l_0(t).$$
Since
$$x_k\in T(b, t\cC)\cap \N\subset \left[1, \left\lfloor t\frac{b-2}{b-1}\right\rfloor\right]\cap
\N\subset [0, mb)$$
for $1\le k\le O_b(t)$,
by \eqref{eq4.1} we have  $O_b(t)\le m\#\cC$, which contradicts the assumption. \ez
\begin{prop}\label{propadadr}
If $$O_b(t)>\min_{k\ge 0}(\#\cC)^{k+1}l_0(b^{-k}t),$$ where $l_0(t)$ is given in Proposition
\ref{prop4.6},
 then $t$ is not a primitively non spectral eigenvalue of $(\mu_{b, \D}, \La(b, \cC))$.
\end{prop}
\proof Denote $T=T(b, t\cC)$. Then
\begin{eqnarray*}
  T&=&\bigcup_{c\in\cC}\phi_c(T)=\bigcup_{I\in\cC^k}\phi_I(T)\subset\bigcup_{I\in\cC^k}\phi_I\left([0,
  t\frac{b-2}{b-1}]\right)\\
  &=&\bigcup_{I\in\cC^k}\left[\phi_I(0), \phi_I(0)+b^{-k}t\frac{b-2}{b-1}\right],
\end{eqnarray*}
where
\[\phi_I(x)=\frac 1{b^k}x+\phi_I(0), \quad \phi_I(0)=c_1b^{-1}+c_2b^{-2}+\cdots+c_kb^{-k}\]
if $I=c_1c_2\cdots c_k\in\cC^k$. It is easy to see that both $\phi_I(0)$ and
$\phi_I(0)+b^{-k}t\frac{b-2}{b-1}$ are not integers. Then by the same idea with the above
\begin{eqnarray*}
O_b(t)=\#\{x_k: x_k\in T\cap \N\}\le (\#\cC)^{k+1}l_0(b^{-k}t)
\end{eqnarray*}
for $k\ge 1$. Hence the assertion follows by the assumption and Proposition \ref{prop4.6}. \ez

The following theorem is an immediate consequence of Proposition \ref{propadadr}.
\begin{thm}\label{theo4.8}
Let $t\in\N$ with $\gcd(t, b)=1$. If any proper $t'\mid t$ is a spectral eigenvalue and
\[O_b(t)>\min_{k\ge 1}(\#\cC)^{k+1}l_0(b^{-k}t),\]
then $t$ is a spectral eigenvalue.
\end{thm}
\begin{thm}
If Artin's conjecture holds, then Conjecture \ref{conj4.2} holds for each non-perfect suqare $b$.
\end{thm}
\proof By Artin's conjecture, there are infinitely many primes $p$ such that $O_b(p)=p-1$. For such
$p$ one has $O_b(p)>\frac {\#\cC}{b}p$. Hence the result follows by Proposition \ref{prop4.6} and
Theorem \ref{theo4.8}.\ez
\begin{thm}
If Artin's conjecture holds and $2^k\#\cC\le b=\al^{2^k}$ for some non-perfect square $\al$,  then
Conjecture \ref{conj4.2} holds.
\end{thm}
\proof
Let $b=\al^{2^k}$ and $\al$ is not a perfect square and let $t$ be a prime with $O_\al(t)=t-1$. Then
\[O_b(t)=O_{\al^{2^k}}(t)=\frac{O_\al(t)}{\gcd(O_\al(t), 2^k)}=\frac{t-1}{\gcd(t-1, 2^k)}.\]
Denote $t=1+2^ms$ where $s$ is odd. Then
 \begin{eqnarray*}
   O_b(t)=\frac{t-1}{2^{\min\{m,k\}}}\ge 2^{-k}(t-1).
 \end{eqnarray*}
 By a direct calculation one has for $t>b$
 \begin{eqnarray*}
   O_b(t)\geq 2^{-k}(t-1)\ge \frac{\#\cC(b-1)}{b^2}t.
 \end{eqnarray*}
 The assertion follows by Proposition \ref{prop4.6} and Theorem \ref{theo4.8}.
\ez

Recall
\[\ell_b(t)=\max\{k\geq1: t^k\mid (b^{O_b(t)}-1)\},\]
and let $p$ be a prime, we have
\[O_b(p^n)=p^{\max\{0, n-\ell_b(p)\}}O_b(p).\]
Then for $m>n$ we have
\begin{eqnarray}  \label{eq4.2}
\frac{O_b(p^m)}{O_b(p^n)}=\frac{p^{\min\{0, m-\ell_b(p)\}}}{p^{\min\{0, n-\ell_b(p)\}}}=\left\{
                                                                                            \begin{array}{ll}
                                                                                              p^{m-n},
                                                                                              &
                                                                                              \hbox{if
                                                                                              $n\ge\ell_b(p)$;}
                                                                                              \\
                                                                                             p^{m-\ell_b(p)},
                                                                                             &
                                                                                             \hbox{if
                                                                                             $m>\ell_b(p)$
                                                                                             and
                                                                                             $n\le\ell_b(p)$;}
                                                                                             \\
                                                                                              1, &
                                                                                              \hbox{if
                                                                                              $m\le\ell_b(p)$
                                                                                              .}
                                                                                            \end{array}
                                                                                          \right..
\end{eqnarray}
\begin{thm}\label{theo4.11}
Let $(b, \D, \cC)$ be a canonical Hadamard triple. For $n,m\in \N$ with $n<m$. Let $p>2$ be a prime
with
\begin{eqnarray}\label{eq4.3}
  \frac{O_b(p^m)}{O_b(p^n)}=M\ge \frac{\#\cC}b\left(\frac{b-2}{b-1}p^{m-n}+b-\#\cC\right).
\end{eqnarray} Then $p^m$ is not a primitive non spectral eigenvalue of  $(\mu_{b, \D}, \La(b,
\cC))$.
\end{thm}
\proof If $p^m$ is a primitive non spectral integer eigenvalue, then by Theorem \ref{theobasic} there
exists $x\in T(b, p^m\cC)\cap\N$ such that
\begin{eqnarray*}
  \left\{b^kx\right\}_{k=1}^{O_b(p^m)}\pmod {p^m}\subset T(b, p^m\cC)\cap\N.
\end{eqnarray*}
Consider
\begin{eqnarray*}
  \left\{b^kx\right\}_{k=1}^{O_b(p^n)}\pmod {p^n}.
\end{eqnarray*}
Note that $0\not\in \left\{b^kx\right\}_{k=1}^{O_b(p^n)}$ and all of them are pairwise different.
Then there exists $j$ such that $b^jx\pmod {p^n}\ge O_b(p^n)$. Denote $$y^*=b^jx\pmod {p^n}\ge
O_b(p^n).$$
Set
$$y_k=b^{j+kO_b(p^n)}x\pmod {p^m}, \qquad \text{for $0\le k\le M-1$}.$$
Then
$$y_k=b^{j+kO_b(p^n)}x\pmod {p^n}=b^{j}x\pmod {p^n}=y^*+\al_kp^n,$$
where all $\al_k\ge 0$ are pairwise different for $0\le k\le M-1$. Take $g\in\N$ with $gp^n=1\pmod
b$. By the earlier part of the proof of Proposition \ref{prop4.6} we have $y_k\in-t\cC\pmod b$. Then
 \begin{eqnarray*}
   \al_k= gy_k-gy^*\pmod b\in -gy^*-gt\cC\pmod b.
 \end{eqnarray*}
Hence by \eqref{eq4.1}
\begin{eqnarray*}
\#\{\al_k: x_k\in [0, lb), 0\le k\le M-1\}\le l\#\cC, \qquad \text{for $l\ge 1$}.
\end{eqnarray*}
Write $M=s\#\cC+r$ for $0\le r<\#\cC$. When $r=0$, then the above holds for $s=M/\#\cC$. Denote
$\al_{max}=\max_{0\le i\le O_b(t)}\al_i$. Then
\[\al_{max}\ge (s-1)b+\#\cC=(\frac M{\#\cC}-1)b+\#\cC=\frac b{\#\cC}M-b+\#\cC.\]
When $1\le r<\#\cC$, then
\[\al_{max}\ge sb+r=\frac{M-r}{\#\cC}b+r=\frac b{\#\cC}M-\frac b{\#\cC}r+r\ge \frac b{\#\cC}M-(\frac
b{\#\cC}-1)(\#\cC-1).\]
Consequently $\al_{max}\ge \frac b{\#\cC}M-b+\#\cC.$ Then
\begin{eqnarray*}
  y_{max}&:=&y^*+\al_{max}p^n\ge O_b(p^n)+p^n(\frac b{\#\cC}M-b+\#\cC)\\
  &=&p^{\min\{0, n-\ell_b(p)\}}O_b(p)+p^n(\frac b{\#\cC}M-b+\#\cC)\\
  &>&p^n(\frac b{\#\cC}M-b+\#\cC)\ge \frac{b-2}{b-1}p^{m},
\end{eqnarray*}
which contradicts the the assumption.\ez

\begin{lem}\label{lem4.12}
If $p^{m}\ge (b-1)(b-\#\cC)$, then
\[\frac{\#\cC}b\left(\frac{b-2}{b-1}p^{m}+b-\#\cC\right)\le \frac{\#\cC}bp^{m}.\]
\end{lem}
\proof
The proof is straightforward.
\ez
\begin{thm}\label{themmlk}
Let $(b, \D, \cC)$ be a canonical Hadamard triple. Let $p$ be a prime larger than $2$ satisfying
$p\ge (b-1)(b-\#\cC).$
If $p^{\ell_b(p)}$ is a spectral integer eigenvalue of $(\mu_{b, \D}, \La(b, \cC))$, then all $p^n$
are complete.
\end{thm}
\proof If $p^{m}$ is not a spectral eigenvalue, then $m>\ell_b(p)$. Without loss of generality we
assume that $p^{m}$ is primitive incomplete. Since
\begin{eqnarray*}
  \frac{O_b(p^m)}{O_b(p^{\ell_b(p)})}=p^{m-\ell_b(p)}>\frac{\#\cC}bp^{m-\ell_b(p)},
\end{eqnarray*}
By Lemma \ref{lem4.12} and Theorem \ref{theo4.11},  $p^{m}$ is not primitive incomplete. This
contradiction yields the assertion.\ez

By Theorem \ref{themmlk}, we immediately have the following:
\begin{coro}
If there is a prime $p$ larger than $2$ such that $p\ge (b-1)(b-\#\cC)$ and $p^{\ell_b(p)}$ is a
spectral integer eigenvalue of $(\mu_{b, \D}, \La(b, \cC))$, then Conjecture \ref{conj4.2} holds.
\end{coro}


\begin{thebibliography}{9999}
\bibitem{AH26}
L. X. An and X. G. He, {\it unpublished note}.
\bibitem{artin65}
E. Artin, \textit{Collected Papers}, Reading, MA: Addison-Wesley, 1965.
\bibitem{BBM82}
J. Bellissard, D. Bessis, and P.Moussa, {\it Chaotic states of almost periodic Schr\"odinger
operators},
Phys. Rev. Lett. {\bf 49} (1982), 701--704.


\bibitem {Dai12}
 X. R. Dai, {\it When does a Bernoulli convolution admit a spectrum?},  Adv. Math. {\bf 231} (2012),
 1681--1693.
 \bibitem{Dai16}
 X. R. Dai, {\it Spectra of Cantor measures}, Math. Ann. \textbf{366} (2016), no. 3-4, 1621--1647.

\bibitem {DHL13}
X. R. Dai, X. G. He and C. K. Lai, {\it Spectral property of Cantor measures with consecutive
digits},  Adv. Math. {\bf 242} (2013), 187--208.


\bibitem {DHS09}
D. Dutkay, D. Han, Q. Sun, {\it On the spectra of a Cantor
measure}, Adv. Math. {\bf 221} (2009), 251--276.

\bibitem{DJ06}
 D. Dutkay and P. Jorgensen, {\it Iterated function systems, Ruelle operators, and invariant
 projective measures}, Math. Comp. {\bf 75} (2006), no. 256, 1931--1970.

\bibitem{DHS14}
D. Dutkay, D. G. Han and Q. Y. Sun, {\it Divergence of the mock and scrambled Fourier series on
fractal measures}, Trans. Amer. Math. Soc. {\bf 366} (2014), 2191--2208.

\bibitem{DH16}
D. Dutkay and J. Hausserman, {\it Number theory problems from the harmonic analysis of a fractal.} J.
Number Theory  {\bf 159} (2016), 7--26.

\bibitem{DK18}
D. Dutkay and I. Kraus, {\it Scaling of spectra of Cantor-type measures and some number theoretic
considerations}, Analysis Math.  {\bf 44} (2018), 335--367.

\bibitem{DHL19}
D. Dutkay, J. Haussermann and  C. K. Lai, {\it Hadamard triples generate self-affine spectral
measures.} Trans. Amer. Math. Soc.  {\bf 371}  (2019),  no. 2, 1439--1481.


\bibitem {Fa90}
 K. Falconer, {\it Fractal Geometry, Mathematical Foundations and Applications}, Wiley, New York,
 1990.

\bibitem {Fu74}
B. Fuglede, {\it Commuting self-adjoint partial differential operators and a group theoretic
problem},  J. Funct. Anal. {\bf 16} (1974), 101--121.

\bibitem {FHW18}
Y. S. Fu, X. G. He and Z. X. Wen, {\it Spectra of Bernoulli convolutions and random convolutions},
J. Math. Pures Appl. {\bf 116} (2018), 105--131.

\bibitem  {HLL13}
X. G. He, C. K. Lai and K. S. Lau, {\it Exponential spectra in $L^2(\mu)$}, Appl. Comput. Harmon.
Anal. {\bf34} (2013),  327--338.

\bibitem{HTW19}
X. G. He, M. W. Tang and Z. Y. Wu, {\it Spectral structure and spectral eigienvalve problems of a
class of self-similar spectral measures}, J. Funct. Anal. {\bf 277} (2019), 3688--3722.

\bibitem{IR90}
K. Ireland and M. Rosen, {\it A classical introduction to modern number theory}, 2nd ed., Graduate
Texts in Mathematics, vol.~84, Springer, New York, 1990.
\bibitem{JP98}
P. Jorgenson and S. Pederson, {\it Dense analytic subspaces in fractal $L^2$-spaces},  J. Anal. Math.
{\bf75} (1998), 185--228.
\bibitem {La01}
 I. {\L}aba, {\it Fuglede's conjecture for a union of two intervals}, Proc. Amer. Math. Soc. {\bf129}
 (2001),  2965--2972.

\bibitem  {LW02}
I. {\L}aba and Y. Wang, {\it On spectral Cantor measures}, J. Funct. Anal. {\bf 193} (2002),
409--420.

\bibitem{LW06}
I. {\L}aba and Y. Wang, {\it Some properties of spectral measures},
Appl. Comput. Harmon. Anal. {\bf 20} (2006),  149--157.
\bibitem{Min57}
H. Minkowski, {\it Diophantische Approximationen: Eine Einf\"{u}hrung in die Zahlentheorie},
Chelsea Publishing (New York, 1957).

\bibitem {Str00}
R. Strichartz, {\it Mock Fourier series and transforms associated with certain cantor measures}, J.
Anal. Math. {\bf81} (2000), 209--238.

\bibitem{Str06}
R. Strichartz, {\it Convergence of Mock Fourier series}, J. Anal. Math. {\bf99} (2006), 333--353.
\end{thebibliography}
\end{document}